\documentclass[a4paper,11pt]{amsart}

\usepackage[
  margin=30mm,
  marginparwidth=25mm,     % + <- Width of your marginpar
  marginparsep=2mm,       % + <- Gap between text block and marginpar
  bottom=25mm,
  ]{geometry}

\usepackage{amsmath,amssymb}
\usepackage[english]{babel}
  \usepackage{paralist}

  \usepackage{graphics} %% add this and next lines if pictures should be in esp format
  \usepackage{epsfig} %For pictures: screened artwork should be set up with an 85 or 100 line screen
\usepackage{graphicx}
\usepackage{epstopdf}%This is to transfer .eps figure to .pdf figure; please compile your paper using PDFLeTex or PDFTeXify.
\usepackage[colorlinks=true]{hyperref}
\usepackage{array}
\usepackage{xcolor}
\usepackage{esint}
\usepackage{latexsym,amsfonts}
\usepackage{amsthm}

\usepackage{verbatim}

\usepackage{bbm}

\usepackage{array}
\newtheorem{thm}{Theorem}[section]
\newtheorem{cor}{Corollary}[section]

\newtheorem{lem}{Lemma}[section]

\theoremstyle{definition}

\newtheorem*{theorem*}{Theorem}

\numberwithin{equation}{section}

\newcommand{\K}{\mathbb K}
\newcommand{\T}{\mathbb T}
\newcommand{\1}{\mathbbm 1}
\newcommand{\R}{\mathbb R}
\newcommand{\N}{\mathbb N}

\newcommand{\bigint}{\begin{picture}(10,10)
\put(-1,2){\line(1,0){10}}
\end{picture}\kern-14pt\int}

\newcommand{\sign}{\text{sign}}

\def\XXint#1#2#3{{\setbox0=\hbox{$#1{#2#3}{\int}$}
     \vcenter{\hbox{$#2#3$}}\kern-.5\wd0}}

\begin{document}

\title[Oscillation results for H\"older and Zygmund functions]
	  {Oscillation of generalized differences of H\"older and Zygmund functions}

\author[A. J. Castro]{Alejandro J. Castro}
\author[J. G. Llorente]{Jos\'e G. Llorente}
\author[A. Nicolau]{Artur Nicolau}

\address{\newline
        Alejandro J. Castro \newline
        Department of  Mathematics, Nazarbayev University, \newline
		010000 Astana, Kazakhstan}
\email{alejandro.castilla@nu.edu.kz}

\address{\newline
        Jos\'e G. Llorente, Artur Nicolau \newline
        Departament de Matem\`{a}tiques,  Universitat Aut\`onoma de Barcelona, \newline
        08193 Bellaterra, Spain}
\email{jgllorente@mat.uab.cat, artur@mat.uab.cat}

\thanks{The first author is partially supported by Spanish Government grant
MTM2016-79436-P. The last two authors are partially supported by the Generalitat de Catalunya, grant 2014SGR 75, and the Spanish Ministerio de Econom\'ia, grant MTM2014-51824-
P}


\footnotetext{Updated: \today.}

\begin{abstract}
In this paper we analyze the oscillation of functions having derivatives in the H\"older or Zygmund class in terms of generalized differences and prove that its growth is governed by a version of the classical Kolmogorov's Law of the Iterated Logarithm. A better behavior is obtained for functions in the Lipschitz class via an interesting connection with Calder\'on-Zygmund operators.
\end{abstract}

\maketitle

%---------------------------------------------------------------------------
\section{Introduction and main results}\label{sect1}

% This paper continues previous work of two of the authors (\cite{LlN}) on the quantitative oscillation properties of H\"older continuous functions on the real line.

\

We introduce the function spaces that will be used in the paper. For $0< \alpha < 1$,  $\Lambda_{\alpha}(\R^d ) $ denotes the space of { \it H\"older continuous} functions of exponent $\alpha $, that is, those $f: \R^d \to \R$ such that 
$$
\|f\|_{\alpha} :=\sup \frac{|f(x + h ) - f(x)|}{|h|^{\alpha}}  < \infty,
$$
where the supremum is taken over all $x\in \R^d$ and all $h\in \R^d \setminus \{ 0 \}$. The case $\alpha = 1$, will require an special treatment. We define the {\it Zygmund class} $\Lambda_{1}(\R^d )$  consisting of all continuous functions $f: \R^d \to \R$ such that
$$
\|f\|_{1} := \sup\frac{|f(x+h) + f(x-h) - 2f(x)|}{|h|} < \infty 
$$
and the {\it Lipschitz class}  $\text{Lip}(\R^d )$, consisting of those $f: \R^d \to \R$ for which 
$$
 \|f\|_{\text{Lip}} := \sup \frac{|f(x+h) - f(x)|}{|h|} < \infty,
$$
the supremum being taken over all $x\in \R^d$ and all $h \in \R^d \setminus \{ 0 \}$ in both cases.  

 While the Lipschitz class coincides formally with $\Lambda_{\alpha}$ for $\alpha =1$, it is more convenient in this paper to reserve the notation $\Lambda_1$ for the Zygmund class, the natural substitute of the Lipschitz class in many problems in Analysis. Observe that $\displaystyle \text{Lip}(\R^d ) \subset \Lambda_{1}(\R^d ) \subset \Lambda_{\alpha}(\R^d )$ if $0< \alpha < 1$ and that the modulus of continuity of Zygmund functions is $\displaystyle O \big( t \log (1/t)\big)$ (\cite{S}). \\

Now let $m\geq 0$ be an integer and $0< \alpha \leq 1$. We define $C^{m, \alpha}(\R^d )$ as  the space of functions $f: \R^d \to \R$ such that $f$ is $m$ times differentiable in $\R^d$ and all the derivatives of $f$ of order $m$ belong to $\Lambda_{\alpha}(\R^d )$. We say that $k= (k_1 ,..., k_d )$ is a multiindex if $k_i \geq 0$ are integers for $i=1, \cdots , d$ and we call $|k| := k_1 + \cdots + k_d$ the degree of $k$. For $x= (x_1 ,..., x_d )\in \R^d$ we use the standard notations $\displaystyle x^k = x_1^{k_1}\cdots x_d^{k_d} $ and 
$$
\partial^k  := \frac{\partial^{k_1}}{\partial x_1^{k_1}}\cdots \frac{\partial^{k_d}}{\partial x_d^{k_d}}
$$
for the higher derivatives. If $f\in C^{m, \alpha}(\R^d )$, we denote 
$$
\|f\|_{m, \alpha} := \sum_{|k|= m} \|\partial^k f \|_{\alpha}.
$$

\

In this paper we will study oscillation properties of functions in the spaces $C^{m, \alpha}(\R^d )$ in terms of generalized differences. We first recall some classical facts and previous results. While a Lipschitz function is differentiable almost everywhere by a classical result of Rademacher (\cite{Rad}), the situation can change dramatically for H\"older continuous functions and Zygmund functions, even if $d=1$. Hardy showed in \cite{H} that  if $b>1$ and  $0< \alpha \leq 1$ then the Weierstrass function $f_{b, \alpha} : \R \to \R $ given by
$$
f_{b, \alpha} (x) := \sum_{k=0}^{\infty} b^{-\alpha k}\cos (b^k x)
$$ 
satisfies $f_{b, \alpha} \in \Lambda_{\alpha}(\R )$ and $f_{b, \alpha}$ is nowhere differentiable. See also \cite{S}, pag. 149. If $0 < \alpha < 1$, Hardy  actually proved an stronger result: 
\begin{equation}\label{hardy}
\limsup_{h\to 0} \frac{|f_{b, \alpha} (x + h) - f_{b, \alpha} (x)|}{|h|^{\alpha}} > 0,
\end{equation}
for each $x\in \R$. 
% Note, however, that (\ref{hardy}) does not exclude  possible cancellations of the differences $f_{b, \alpha }(x+h)-f_{b, \alpha} (x)$ as $h\to 0$.      

\

Inspired by earlier work of Lyubarskii and Malinnikova (\cite{LM}), the last two authors of the present paper   introduced in \cite{LlN} a quantitative way of measuring the oscillation of H\"older continuous functions. Let $f\in \Lambda_{\alpha}(\R )$, with $0< \alpha <1$. For $0 < \varepsilon < 1/2$ and $x\in \R$, define 
\begin{equation}\label{Theta0}
\Theta_{\varepsilon}f(x) := \int_{\varepsilon}^1 \frac{f(x+h) - f(x-h)}{h^{\alpha}} \frac{dh}{h}.
\end{equation}
Observe that  $\displaystyle \|\Theta_{\varepsilon}f \|_{\infty} \leq C \log  ( 1/\varepsilon  )$ and that such global bound cannot be improved, as the function $f(x) = |x|^{\alpha} sgn(x)$ shows. However, the growth rate $\log (1 /\varepsilon ) $  can be substantially improved for almost all $x\in \R$. Indeed, the following Law of the Iterated Logarithm was obtained in \cite{LlN}: 

$$
\limsup_{\varepsilon \to 0} \frac{|\Theta_{\varepsilon}f(x)|}{\sqrt{\log  ( 1/\varepsilon )  \log \log  \log \big ( 1/ \varepsilon )}} \leq C(f, \alpha ) < \infty,
$$
for almost all $x\in \R$. This result admits possible extensions in at least three directions: 
\begin{itemize}
\item[i)] higher dimensional analogues and functions of higher order of differentiability;
\item[ii)] the case $\alpha =1$, which was not covered in \cite{LlN}; and  \item[iii)] the use of other asymmetric differences in  \eqref{Theta0} instead of the symmetric difference $f(x+h) - f(x-h)$.   
\end{itemize}
\quad\\
Let $\sigma$ be a (signed) compactly supported Borel measure in $\R^d$ with finite total variation and $\sigma (\R^d)=0$. For a locally integrable function $f : \R^d \to \R$ we define the {\it generalized differences} associated to $\sigma$ as 
\begin{equation}\label{Delta}
\Delta_{\sigma}f(x,h) := \int_{\R^d} f(x + hw) d\sigma (w),
\end{equation}
where $x\in \R^d$ and $h>0$. Observe that, if $d=1$, we recover the usual  symmetric first order  difference $f(x+h) - f(x-h)$ and the symmetric second order difference $f(x+h) + f(x-h) - 2f(x)$ by choosing $\sigma := \delta_{1} - \delta_{-1}$ or $\sigma := \delta_1 + \delta _{-1} -2\delta_{0}$ in \eqref{Delta} respectively. Note that this setting also captures the classical difference operators defined by 
$$\Delta_1 f(x,h) := f(x+h) - f(x)$$ 
and 
$$\Delta_k f(x,h) := \Delta_1 (\Delta_{k-1})f(x,h).$$ 
Actually, it is easy to see that 
$$
\Delta_k f(x,h) = \sum_{j=0}^k (-1)^{k+j} \binom{k}{j} f(x + jh), 
$$
so $\Delta_k = \Delta_{\sigma}$ for the choice 
$$
\sigma := \sum_{j=0}^k (-1)^{k+j} \binom{k}{j} \delta_{j}.
$$

The first and second order symmetric differences have been extensively used in  Real Analysis and there are a number of beautiful classical results on the interaction between the symmetric differences and the usual derivatives. One of the most celebrated is the theorem of Khintchine (\cite{K}) according to which a measurable function $f:\R \to \R $ is differentiable at almost every point $x$  at which 
$$
\limsup_{h\to 0} \frac{f(x+h)-f(x-h)}{2h} < +\infty.
$$
See \cite{Th} for this and other
related results. A version of Khintchine's theorem for certain asymmetric differences was obtained by Valenti (\cite{V}). \\

Now, if $f\in C^{m, \alpha}(\R^d )$, and $\sigma$ and $\Delta_{\sigma}$ are as above, we define the {\it oscillation function} associated to $f$ and $\sigma$ as 
\begin{equation}\label{Theta}
\Theta_{\varepsilon}^{\sigma} f(x) := \int_{\varepsilon }^1 \frac{\Delta_{\sigma}f(x,h)}{h^{m+\alpha}} \frac{dh}{h},
\end{equation}
for $x\in \R^d$ and $0< \varepsilon < 1/2$. From Lemma \ref{lemmawellknown} in Section \ref{sect2} below, it follows that $f \in C^{m,\alpha}$  if and only if there exists a constant $C>0$ such that $\sup \{|\Delta_{\sigma}f (x,h) | : x \in \R^d\} \leq C \|\sigma \| h^{m+\alpha } $ 
% $\displaystyle \Delta_{\sigma}f (x,h)  = O ( h^{m+\alpha }) $ 
for any compactly supported $\sigma $ of finite total variation with vanishing moments of order at most $[m + \alpha ]$. Here $\|\sigma\|$ denotes the total variation of $\sigma$ in $\R^d$. In particular  $\displaystyle \|\Theta^{\sigma}_{\varepsilon} \|_{\infty} = O \big (\log ( 1 / \varepsilon ) \big ) $ for such $\sigma$'s  and, as mentioned above, this estimate is sharp. The following theorem, the main result of this paper, shows that, as in \cite{LlN}, this global bound can be substantially improved by means of the corresponding Law of the Iterated Logarithm. 

\begin{thm}\label{main}
Let $m\geq 0$ be an integer, $0< \alpha \leq 1$ and $f\in C^{m, \alpha}(\R^d )$. Suppose that $\sigma$ is a compactly supported (signed) Borel measure on $\R^d$ of finite total variation such that 
\begin{equation}\label{momentos}
\int_{\R^d} x^k d\sigma (x) = 0,
\end{equation}
for any multiindex $k$ with $ 0 \leq |k| \leq [m+\alpha]$. Let $\Theta^{\sigma}_{\varepsilon} f$ be as in \eqref{Theta}. Then, there exits a  constant $C = C(m,d, \sigma)>0$ such that 
$$
\limsup_{\varepsilon \to 0} \frac{|\Theta^{\sigma}_{\varepsilon}f(x)|}{\sqrt{\log \big(1/\varepsilon \big) \log \log \log \big( 1/\varepsilon\big) }} \leq C  \|f\|_{m, \alpha} \, , 
$$
for almost all $x\in \R^d$. 
\end{thm}

For particular choices of the measure $\sigma$ we get the following immediate consequences. 

\begin{cor}\label{1}
Let $0 < \alpha < 1$ and $f\in \Lambda_{\alpha}(\R^d ) $. Let $p\in \N$ , $a_1 , ..., a_p \in \R^d$ and $\mu_1 , ..., \mu_p \in \R$ such that $\displaystyle \sum_{i=1}^p \mu_i  = 0$. Define  
$$
\Gamma_{\varepsilon}f(x) := \int_{\varepsilon}^1 \frac{\sum_{i=1}^p \mu_i f(x+ a_i h) }{h^{\alpha}} \frac{dh}{h},
$$
for $x\in \R^d$. Then, there exists a positive constant $C$ depending only on $d$, $\alpha$, $\|f\|_{\alpha}$ and $a_1$,...$a_p$ such that 
$$
\limsup_{\varepsilon \to 0} \frac{|\Gamma_{\varepsilon}f(x)|}{\sqrt{\log \big(1/\varepsilon \big) \log \log \log \big( 1/\varepsilon\big) }} \leq C ,
$$
for almost all $x\in \R^d$. 
\end{cor}

\begin{cor}\label{2}
Let $f\in \Lambda_1 (\R^d )$. Let $p\in \N$ , $a_1 , ..., a_p \in \R^d$ and $\mu_1 , ..., \mu_p \in \R$ such that $\displaystyle \sum_{i=1}^p \mu_i  =  \sum_{i=1}^p \mu_i a_i = 0$. Define  
$$
\Omega_{\varepsilon}f(x) := \int_{\varepsilon}^1 \frac{\sum_{i=1}^p \mu_i f(x+ a_i h) }{h} \frac{dh}{h},
$$
for $x\in \R^d$. Then, there exists a positive constant $C$ depending only on $d$, $a_1$, ...,$a_p$ and $\|f\|_{1}$  such that
$$
\limsup_{\varepsilon \to 0} \frac{|\Omega_{\varepsilon}f(x)|}{\sqrt{\log \big(1/\varepsilon \big) \log \log \log \big( 1/\varepsilon\big) }} \leq C,
$$
for almost all $x\in \R^d$. 
\end{cor}

The following application follows from the choice $\sigma = \omega - \delta_0$, where $\omega$ is  the normalized surface measure on the unit sphere $\mathbb{S}^{d-1}$ of $\R^d$.

\begin{cor}
Let $0 < \alpha \leq 1$, $f\in \Lambda_{\alpha}(\R^d )$ and $\omega$
be the normalized surface measure on the unit sphere $\mathbb{S}^{d-1}$ of $\R^d$. Define 
$$
\mathcal{M}_{\varepsilon}f(x) := \int_{\varepsilon}^1 \int_{\mathbb{S}^{d-1}} [f(x + h\xi ) - f(x)]d\omega ( \xi ) \frac{dh}{h^{\alpha +1}}.
$$
Then, there exists a positive constant $ C $ depending only on $d$, $\alpha$ and $\|f\|_{\alpha}$ such that 
$$
\limsup_{\varepsilon \to 0} \frac{|\mathcal{M}_{\varepsilon}f(x)|}{\sqrt{\log \big(1/\varepsilon \big) \log \log \log \big( 1/\varepsilon\big) }} \leq C .
$$
\end{cor}

In the case $d=1$, $0< \alpha < 1$, $p=2$, $a_1 =\mu_1= 1$ and $a_2 =\mu_2 =-1$, which corresponds to the oscillation function given by \eqref{Theta0}, the sharpness of Corollary \ref{1}  was noted in \cite[Section 5]{LlN}. As for the case $d=1$, $\alpha =1$, $p=3$, $a_1 =1$, $a_2 = -1$, $a_3 = 0$, 
$\mu_1=\mu_2=1$ and $\mu_3=-2$, which corresponds to the oscillation function
$$
\Upsilon_{\varepsilon} f(x) := \int_{\varepsilon}^{1} \frac{f(x+h) + f(x-h) - 2f(x) }{h} \, \frac{dh}{h},
$$
the sharpness of Corollary \ref{2} can be also established in the same way. Indeed, take $b>1$ and define the Weierstrass-type function 
$$
f(x) := \sum_{k=1}^{\infty} b^{-k} \cos (b^k x).
$$
It is well known that $f\in \Lambda_1 ( \R )$. Elementary computation shows that 
$$
\Upsilon_{\varepsilon}f(x) = \sum_{k=1}^{\infty} a_k (\varepsilon )\cos (b^k x),
$$
where
$$
a_k (\varepsilon ) := -2 \int_{b^k \varepsilon}^{b^k} \frac{1- \cos t}{t^2} dt.
$$
It can be shown by direct computation that there exists a constant $C(b) >0$ such that 
$$
\Big| \Upsilon_{\varepsilon}f(x) - \sum_{k=1}^{N(\varepsilon )}a_k (0)\cos (b^k x) \Big| \leq  C(b), \quad x \in \R, 
$$
where $N(\varepsilon ) $ is the smallest integer $n$ such that $\varepsilon b^n \geq 1$. Since $\displaystyle \lim_{k\to \infty} a_k (0) > 0$, the Law of the Iterated Logarithm for lacunary  trigonometric series (\cite{W}) shows that 
$$
\limsup_{\varepsilon \to 0} \frac{\Upsilon_{\varepsilon}f(x)}{\sqrt{\log \big(1/\varepsilon \big) \log \log \log \big( 1/\varepsilon\big) }} >0,
$$
for almost all $x\in \R$. \\

It is worth mentioning that Theorem \ref{main} holds due to certain cancellations which occur in the oscillation function $\Theta^{\sigma}_{\varepsilon}f$. It was already proved in  \cite{LlN} that for any $0< \alpha < 1$ there exists $f \in \Lambda_{\alpha} (\R)$ such that for almost every $x \in \R$ and for any $0< \varepsilon < 1/2$, we have 
\begin{equation}
 \int_{\varepsilon}^1  \left|\frac{f(x+h) - f(x-h)}{h^{\alpha}} \right| \frac{dh}{h} > \log (1/ \varepsilon) . 
\end{equation} 
In \cite{LlN} the LIL for the (symmetric) oscillation of Holder continuous functions was deduced from a subgaussian estimate which was proved in two steps. First dyadic martingales were used to obtain a discrete version of the subgaussian estimate and then an averaging technique due to Garnett and Jones was applied to transfer the result in the discrete setting to the continuous one.
Our approach now is more direct and simple. The main idea of the proof of Theorem \ref{main} is to approximate the oscillation function $\Theta^{\sigma}_{\varepsilon}f$, up to a bounded term, by a dyadic martingale with uniformly bounded increments. Then Theorem \ref{main} will follow  from the Law of the Iterated Logarithm (LIL) for such class of  martingales. The LIL is sharp at almost every point as the previous example with the Weierstrass type function shows. 
However, since martingales in this class are bounded at a set of maximal Hausdorff dimension (see \cite{M} for the one dimensional case $d=1$ and \cite{Ll} or  \cite{N} for $d>1$ ), the approximation also gives the following result.

\begin{cor}\label{bounded}
Let $m\geq 0$ be an integer, $0< \alpha \leq 1$ and $f\in C^{m, \alpha}(\R^d )$. Suppose that $\sigma$ is a compactly supported (signed) Borel measure on $\R^d$ of finite total variation such that \eqref{momentos} holds. 
 Let $\Theta^{\sigma}_{\varepsilon} f$ be as in \eqref{Theta}. Then, the set 
$$
\{x \in \R^d : \sup_{0 < \varepsilon < 1} |\Theta^{\sigma}_{\varepsilon} f (x)| < \infty \}
$$
has Haussdorff dimension $d$. 
\end{cor}

When acting on Lipschitz functions, the generalized oscillation operators introduced in this paper have a better behavior and also an interesting connection with Calder\'on-Zygmund theory as the following result shows. See \cite{G} or Section \ref{sect4} for definitions. 
\begin{thm}\label{CZthm}
Suppose that $d=1$, $f\in Lip (\R )$ with compact support and $\sigma$ is a (signed) compactly supported Borel measure on $\R$ with finite total variation such that \eqref{momentos} holds for $k=0,1$. Define 
\begin{equation}\label{CZ}
\widetilde{\Theta}_{\varepsilon}^{\sigma}f(x) := \int_{\varepsilon}^1 \Delta_{\sigma}f(x,h)\frac{dh}{h^2},
\end{equation}
for $x\in \R$. Then there exists a Calder\'on-Zygmund Kernel $K_0$ such that
$$
\sup_{0 < \varepsilon < 1} |\widetilde{\Theta}_{\varepsilon}^{\sigma}f(x) - \int_{ |t| > \varepsilon  M} K_0(t) f'(x-t) dt | 
$$
is uniformly bounded for $x \in \R$. In particular 
$$
\sup_{0 < \varepsilon < 1} | \widetilde{\Theta}_{\varepsilon}^{\sigma}f(x) | <  \infty,
$$
for a.e. $x\in \R$. 
\end{thm}
% \blue{Philosophically, Theorem \ref{CZthm} evokes the well known result on the boundedness of the Hilbert transform for Lipschitz functions in $\R$} but we will provide an independent approach here.
It will be shown that the Calder\'on-Zygmund Kernel $K_0$ depends directly on the measure $\sigma$ and the relevant size and cancellation properties of $K_0$ will follow from condition  \eqref{momentos}. It is worth mentioning that when $\sigma = \delta_{1} + \delta_{-1} - 2 \delta_0$, it turns out that $K_0 (t) = sign(t) / t$, $t \neq 0$ and hence in this case $\widetilde{\Theta}_{\varepsilon}^{\sigma}f(x) $ is, up to a uniformly bounded term, the truncated Hilbert transform of $f'$.  

\

The structure of the paper is as follows. Section \ref{sect2} contains some auxiliary lemmas that will be basic in the proof of Theorem \ref{main}. In Section \ref{sect3} we will see  that the oscillation function $\Theta^{\sigma}_{\varepsilon}f$ can be approximated, up to a bounded term, by a dyadic martingale with uniformly bounded increments. Theorem \ref{main} and Corollary \ref{bounded} will follow  easily. Finally, the connection between the operators  $\widetilde{\Theta}_{\varepsilon}^{\sigma}$ and Calder\'on-Zygmund operators in the Lipschitz context is given in Section \ref{sect4}. \\

{ \bf Notation}. We will denote by $m_d$ the Lebesgue $d$-dimensional measure in $\R^d$. All the cubes considered in the paper are understood to have parallel sides to the hyperplane coordinates. We denote by $\ell(Q)$ the side length of the cube $Q$. Two cubes are said adjacent if they have the same side length and a common face. If $\beta>0$, $[\beta]$ stands for the smallest integer less or equal than $\beta$.

%---------------------------------------------------------------------------
\section{Auxiliary Results}\label{sect2}

% A multiindex $k$ is a $d$-tuple of nonnegative integers, that is, $k=(k_1, \ldots , k_d)$ where $k_i  \geq 0$, $i=1, \ldots , d$ are integers. The number $|k| = k_1 + \ldots + k_d$ is called the degree of $k$.  If $x=(x_1 ,..., x_d) \in {\R}^d$ and $k$ is a multiindex as above, we use the notation $x^k = x_1^{k_1} \ldots x_d^{k_d} $. We also use the  standard notation ${\partial }^k f = {\partial }^{k_1} \ldots {\partial }^{k_d} f$. Given a finite measure $\sigma$ in ${\R}^d $ and a locally integrable function $f$ defined in ${\R}^d $, we use the notation 
% $$
% \Delta_{\sigma} f (x,h) = \int_{{\R}^d } f(x+hw) d\sigma (w)
% $$
Our first auxiliary result collects several well known descriptions of  functions in $C^{m,\alpha}({\R}^d)$.

\begin{lem}\label{lemmawellknown}
Let $m \geq 0$ be an integer and $0< \alpha \leq 1$. Let $f: {\R}^d \rightarrow \R$ be a bounded continuous function. The following conditions are equivalent: 
    
    \begin{itemize}
        \item[$a)$] $f \in C^{m,\alpha}({\R}^d) $. 
       
        \item[$b)$] For any integer $\ell > [m+ \alpha]$, there exists a constant $C_\ell >0$ such that $$|\Delta_{\ell}f(x,h)| \leq C_l  h^{m+\alpha}, \quad x \in {\R}^d, \ h>0.$$ 
                
        \item[$c)$] There is a constant $K_1>0 $ such that for any ball $B \subset {\R}^d $ of radius $r(B)$, there exists a polynomial $P_B$ of degree less or equal to $[m + \alpha]$ such that 
        $$ 
        \sup_{x \in B} |f(x) - P_B (x)| < K_1  r(B)^{m+\alpha}.
        $$
        
         \item[$d)$] For any finite compactly supported measure $\sigma$ on ${\R}^d$ for which \eqref{momentos} holds for any multiindex $k$ with $0\leq |k| \leq [m + \alpha]$, there is a constant $K_2>0 $ such that  
$$
|\Delta_{\sigma} f (x,h)| \leq K_2   h^{m + \alpha},  \quad x \in {\R}^d, \ h>0.
$$    
    \end{itemize}
\end{lem}

\begin{proof}
The equivalence between $a)$ and $b)$ is well known and can be found for instance in \cite[page 202]{Tr} or in \cite[Theorem 6.1]{Kr}. The equivalence between $a)$ and $c)$ can be found in \cite[page 13]{JW}.\\

Condition $d)$ implies $b)$ because given a positive integer $\ell$, the measure $\sigma (\ell)$ defined as
$$
\sigma (\ell) = \sum_{j=0}^\ell (-1)^{\ell+j} {{\ell}\choose{j}} \delta_{j}
$$
satisfies 
$$
\int_{{\R}^d } f(x+hw) d\sigma (\ell) (w) = \Delta_{\ell}f(x,h), \quad x\in {\R}^d, \ h>0 .
$$
Moreover if $\ell >[m +\alpha]$, one can check that the first $[m + \alpha]$ moments of  $\sigma (\ell)$ vanish. \\

Now let us see that $c)$ implies $d)$. Let $\sigma$ be a finite measure in ${\R}^d$ whose support is contained in the ball $\{x \in {\R}^d : |x|\leq M \}$. Fix $x \in {\R}^d$, $h>0$ and apply $c)$ to the ball $B$ centered at $x$ of radius $Mh$ to obtain a polynomial $P_B$ of degree smaller or equal to $[m + \alpha]$ such that 
$$\sup_B |f - P_B| \leq K_1 (Mh)^{m + \alpha}.$$
Since the first $[m+ \alpha]$ moments of $\sigma$ vanish, we deduce that
$$
| \Delta_{\sigma} f (x,h) |
	=\Big|\int_{{\R}^d } \Big[ f(x+hw) - P_B (x+hw) \Big] d\sigma (w) \Big| 
    < K_1 (Mh)^{m+\alpha} \| \sigma \|.
$$
\end{proof}

The next three auxiliary results are needed to approximate $\Theta_{\varepsilon}^{\sigma}$ by a dyadic martingale with bounded increments. 

\begin{lem}\label{lemmacrutial1}
Let $m \geq 0$ be an integer, $0< \alpha < 1$ and $f \in C^{m,\alpha}({\R}^d) $. Let $\sigma $ be a compactly supported (signed) Borel measure on $\R^d$ with finite total variation such that \eqref{momentos} holds for all multiindex $k$ with $0 \leq |k| \leq m$. Then:  
\begin{itemize}
 \item[$a)$] There exists a constant $C_1 = C_1 (m,d, \sigma)  >0$ independent of $f$ such that for any $h>0$ and any pair of points $x,t \in {\R}^d$, one has
 $$
 |\Delta_{\sigma} f (x,h) - \Delta_{\sigma} f (t,h)| \leq C_1 \|f\|_{m+\alpha}   |x-t|^{ \alpha}  h^{m}  \, .
 $$
 
\item[$b)$] There exists a constant $C_2= C_2 (m,d, \sigma) >0$ independent of $f$, such that for any multiindex $\kappa$ with $|\kappa|=1$, any $h>0$ and any pair of points $x,t \in {\R}^d$, one has
$$
\Big|\int_{{\R}^d } w^\kappa \Big[ f(x+hw^\kappa) - f(t+hw^\kappa) \Big] d\sigma (w) \Big| 
\leq C_2 \|f\|_{m, \alpha} \,  \, |x-t|^{\alpha} h^{m} .
$$

\item[$c)$] There exists a constant $C_3 = C_3 (m,d, \sigma) >0$ independent of $f$, such that for any cube $Q$ of ${\R}^d$ and any $0<h< \ell(Q)/2$, one has
$$
\Big|\int_Q \Delta_{\sigma} f (x,h) dm_d (x) \Big| \leq C_3 \|f\|_{m, \alpha} \,  \ell(Q)^{d + \alpha -1}  \, h^{m + 1}.
$$
\end{itemize}
\end{lem}

\begin{proof}
% The estimate in $a)$ follows from part $d)$ of Lemma~\ref{lemmawellknown} . 
To prove part $a)$, write Taylor's formula,
$$
f(x+hw)= \sum_{|j| \leq m-1} \frac{{\partial}^j f (x)}{j!} (hw)^j + m \sum_{|j|=m} \frac{(hw)^j}{j!} \int_0^1 (1-s)^{m-1} {\partial}^j f (x+shw) ds \, . 
$$
Since the first $m$ moments of the measure $\sigma$ vanish, we deduce
\begin{align}\label{eq:moment}
 & \int_{{\R}^d }  \Big[ f(x+hw) - f(t+hw) \Big] d\sigma (w) \nonumber
 \\
 & \qquad  =   m \sum_{|j|=m} \int_{{\R}^d }   \frac{(hw)^j}{j!} \int_0^1 (1-s)^{m-1} \Big[{\partial}^j f (x+shw) - {\partial}^j f (t+shw)\Big] ds d \sigma (w).
\end{align}
Since  $|{\partial}^j f (x+shw) - {\partial}^j f (t+shw)| \leq \|f\|_{m , \alpha} |x-t|^{\alpha},$ the estimate in $a)$ follows. \\

The proof of part $b)$ is similar. Write Taylor's formula, 
$$
f(x+hw)= \sum_{|j| \leq m-1} \frac{{\partial}^j f (x)}{j!} (hw)^j + m \sum_{|j|=m} \frac{(hw)^j}{j!} \int_0^1 (1-s)^{m-1} {\partial}^j f (x+shw) ds \, . 
$$
Since the first $m$ moments of the measure $\sigma$ vanish, we deduce
\begin{align}\label{eq:moment}
 & \int_{{\R}^d } w^\kappa \Big[ f(x+hw^\kappa) - f(t+hw^\kappa) \Big] d\sigma (w) \nonumber
 \\
 & \qquad  =   m \sum_{|j|=m} \int_{{\R}^d } w^\kappa  \frac{(hw^\kappa)^j}{j!} \int_0^1 (1-s)^{m-1} \Big[{\partial}^j f (x+shw^\kappa) - {\partial}^j f (t+shw^\kappa)\Big] ds d \sigma (w).
\end{align}
Since  
$|{\partial}^j f (x+shw^\kappa) - {\partial}^j f (t+shw^\kappa)| \leq   \|f\|_{m , \alpha} |x-t|^{\alpha},$ the estimate in $b)$ follows. \\

Let us now prove part $c)$. One can assume that the support of $\sigma$ is contained in the unit ball. Using that $\sigma ({\R}^d) =0$ and Fubini's theorem, we have
$$
\int_Q \Delta_{\sigma} f (x,h) dm_d (x) = \int_{{\R}^d} \int_Q \Big[f(x+hw) - f(x) \Big] d m_d (x) d \sigma (w).
$$
Fix $w=(w_1, \ldots , w_d) \in {\R}^d$, write $\widetilde{w}^j = \sum_{i=1}^j w_i e_i $, $j=1, \ldots , d$ and $\widetilde{w}^0 = (0, \ldots , 0)$. Here $\{e_j : j=1, \ldots , d \}$ is the canonical basis of ${\R}^d$. Then
$$
\int_Q \Delta_{\sigma} f (x,h) dm_d (x) = \sum_{j=1}^d A_j (h) \, ,
$$
where
$$
A_j (h) := \int_{{\R}^d} \int_Q \Big[f(x+h\widetilde{w}^j) - f(x+ h \widetilde{w}^{j-1}) \Big] d m_d (x) d \sigma (w) \, .
$$
Since $ h < \ell(Q) / 2$, a cancellation occurs in the inner integral. Actually consider the cube $Q' = Q + h \widetilde{w}^{j-1}$ and let $Q^{*}$ be the cube in ${\R}^{d-1}$ obtained as intersection of $Q'$ with the hyperplane orthogonal to $e_j$ containing the center of $Q'$, we have 
\begin{align*}
& \int_Q \Big[f(x+h\widetilde{w}^j) - f(x+ h \widetilde{w}^{j-1}) \Big] d m_d (x) \\
& \qquad =   \int_{Q'} \Big[f(x+hw_j e_j) - f(x)\Big] \, d m_d (x) \\
& \qquad = \int_{Q^*} \int_0^1 \Big[ f(x^* + u_j + hw_j s e_j ) - f(x^* - u_j + hw_j s e_j ) \Big] h w_j \, ds \, d m_{d-1} (x^*),
\end{align*}
where $\displaystyle u_j :=  \frac{\ell(Q)}{2} \, e_j $. Hence, 
$$
A_j (h) 
= \int_{Q^*} h \int_0^1 \int_{{\R}^d} w_j  \Big[f(x^* + u_j + hw_j s e_j) - f(x^* - u_j + hw_j s e_j)\Big] d \sigma (w) \, ds \,  d m_{d-1} (x^*) . 
$$
Applying the estimate in $b)$ in the inner integral we deduce that 
$$|A_j (h)| \leq  C_3 \|f\|_{m , \alpha} \|\sigma \| \ell(Q)^{d + \alpha -1}  h^{m + 1}.$$ 
\end{proof}

In the case $\alpha = 1$ we need a slight variation of the previous result whose proof uses the following technical statement. 

\begin{lem}\label{lemmatechnical}
There exists a constant $C=C(d) >0$ only depending on the dimension such that for any function $f \in \Lambda_1 (\R^d)$ and any points $x,t,w \in \R^d$ satisfying $|x-t| > |w|/2$, one has
$$
| f(x+w) - f(x) - (f(t+w) - f(t)) | \leq C \|f \|_{\Lambda_1} |w| \log (1 + |x-t|/ |w|) . 
$$
\end{lem}
\begin{proof}
We can assume that $f$ has compact support. 
Let $u$ be the harmonic extension of $f$ to the upper half space $\R^{d+1}_+ = \{(x,y) : x \in \R^d , y>0\}$. It is well known that there exists a constant $C_1 >0$ such that for any $x, w \in \R^d$ one has
$$
| f(x+w) - f(x) -  \langle \nabla_x u (x,|w|) ,  w \rangle| \leq C_1 \|f \|_{\Lambda_1} |w| . $$
See for instance Proposition 2.3 of \cite{DLlN} . It is also well known that the gradient of $u$ is in the Bloch space and actually there exists a constant $C_2 = C_2 (d) >0$ such that $\sup \{ y |{\partial}^k u (x,y)| : x \in \R^d , y>0 \} < C_2 \|f \|_{\Lambda_1} $ for any multiindex $k$ with $|k|=2$. See \cite{S}, pag. 145. Thus 
$$
|\nabla_x u (x,|w|) - \nabla_x u (t,|w|) | < C_2 \|f\|_{\Lambda_1} \inf \int_{\Gamma} \frac{ds}{y}, 
$$
where the infimum is taken over all rectifiable curves $\Gamma$ in $\R^{d+1}_+$ joining the points $(x,|w|)$ and $(t,|w|)$. The Lemma follows from the estimate
$$
 \inf_{\Gamma} \int_{\Gamma} \frac{ds}{y} \leq C \log (1+ |x-t|/ |w|). 
$$

\end{proof}

We now state the analogue of Lemma \ref{lemmacrutial1} in the case $\alpha=1$.

\begin{lem}\label{lemmacrutial2}
Let $m \geq 0$ be an integer and $f \in C^{m,1}({\R}^d) $. Let $\sigma $ be a compactly supported (signed) Borel measure on $\R^d$ with finite total variation  such that \eqref{momentos} holds for all $k$ with $0 \leq |k| \leq m+1$. Then:  
\begin{itemize}
% \item[$a)$] There exists a constant $C_1 = C_1 (m,d) % >0$ independent of $f$ and $\sigma$ such that 
% $$
% |\Delta_{\sigma} f (x,h)| \leq C_1 \|f\|_{m+1} % |\sigma| ({\R}^d) h^{m+1} \, , x \in {\R}^d , h>0
% $$
\item[$a)$] There exists a constant $C_1 = C_1 (m,d, \sigma) >0$ independent of $f$, such that for any $h>0$ and any pair of points $x,t \in {\R}^d$ such that $|x-t|> h/2$, one has
$$ |\Delta_{\sigma} f (x,h) - \Delta_{\sigma} f (t,h)|
 \leq C_1 \|f\|_{m,1}  \log \Big(\frac{|x-t|}{h} + 1\Big) \, h^{m+1}.
$$

\item[$b)$] There exists a constant $C_2 = C_2 (m,d, \sigma) >0$ independent of $f$, such that for any multiindex $\kappa$ with $|\kappa|=1$, any $h>0$ and any pair of points $x,t \in {\R}^d$ such that $|x-t|> h/2$, one has
$$ \qquad \qquad
\Big|\int_{{\R}^d } w^\kappa \Big[ f(x+hw^\kappa) - f(t+hw^\kappa) \Big] d\sigma (w) \Big| \leq C_2 \|f\|_{m,1}  \log \Big(\frac{|x-t|}{h} + 1\Big) \, h^{m+1}.
$$

\item[$c)$] There exists a constant $C_3 = C_3 (m,d, \sigma) >0$ independent of $f$, such that for any cube $Q$ of ${\R}^d$ and any $0<h< \ell(Q)/2$, one has
$$
\Big|\int_Q \Delta_{\sigma} f (x,h) dm_d (x) \Big | \leq C_3 \|f\|_{m,1}   \ell(Q)^{d  -1} \log \Big(\frac{\ell(Q)}{h}\Big) \, h^{m + 2}.
$$
\end{itemize}
\end{lem}

\begin{proof}

% The estimate in $a)$ follows from part $d)$ of Lemma~\ref{lemmawellknown}. 
Arguing as in the proof of part $a)$ of Lemma \ref{lemmacrutial1} we see that identity \eqref{eq:moment} holds.
Moreover, since the $m+1$ moments of the measure $\sigma$ also vanish, we can replace $\displaystyle {\partial}^j f (x+shw^\kappa)$ (respectively $ \displaystyle {\partial}^j f (t+shw^\kappa)$) by $\displaystyle {\partial}^j f (x+shw^\kappa) - {\partial}^j f (x)$ (respectively $\displaystyle {\partial}^j f (t+shw^\kappa) -  {\partial}^j f (t)$).
% \begin{equation*}
% \begin{split}
%  & \int_{{\R}^d } w^\kappa ( f(x+hw) - f(t+hw)) d\sigma  % (w) =
%  \\*[5pt]
%   =  & m \sum_{|j|=m} \int_{{\R}^d } w^\kappa  % % %  \frac{(hw)^j}{j!} \int_0^1 (1-s)^{m-1} ( {\partial}^j f (x+shw) - {\partial}^j f (x)  - ( {\partial}^j f (t+shw) -  {\partial}^j f (t))) ds d \sigma (w).
% \end{split}
% \end{equation*}
By  Lemma \ref{lemmatechnical} there exists a universal constant $C>0$ such that if $|x-t| > h/2$ and $0<s<1$, one has  
$$
\Big|{\partial}^j f (x+shw^\kappa) - {\partial}^j f (x) - \Big[ {\partial}^j f (t+shw^\kappa) - {\partial}^j f (t) \Big]\Big| 
< C h s \|f \|_{m+1} \log \Big(\frac{|x-t|}{hs} + 1\Big)   .
$$
Then 
\begin{align*}
& \Big|\int_{{\R}^d } w^\kappa \Big[ f(x+hw^\kappa) - f(t+hw^\kappa) \Big] d\sigma (w)\Big| \\
& \qquad \qquad \leq C(m,d) h^{m+1}  \|f \|_{m+1} \int_0^1 s (1-s)^{m-1} \log \Big(\frac{|x-t|}{hs} + 1\Big) ds,
\end{align*}
and estimate $a)$ follows. \\

The proof of parts $b)$ and $c)$ proceeds as in part $b)$ of Lemma \ref{lemmacrutial1}. 
\end{proof}

%---------------------------------------------------------------------------
\section{Reduction to the martingale setting and proof of Theorem \ref{main}.}\label{sect3}
Let $Q_0 = [0,1)^d$ be the unit cube in $\R^d$. Since the problem under consideration is local, we will restrict ourselves to the study of the quantities $\Theta_{\varepsilon}^{\sigma}f(x)$ for $x\in Q_0$ and $f \in C^{m, \alpha}(\R^d )$. 
We will see in this section that it is possible to construct a dyadic martingale $\{ S_n \}$ in $Q_0$ so that the asymptotic behavior of $\Theta_{\varepsilon}^{\sigma}f$ as $\varepsilon \to 0$,  can be transferred to the asymptotic behavior of  $\{ S_n \}$.\\

Denote by $\mathcal {D}_{n}$ the family of all dyadic cubes of $Q_0$ of the generation $n$, that is those 
$$
Q := I_1 \times I_2 \times ...\times I_d,
$$
where $I_j = [m_j 2^{-n}, (m_j +1)2^{-n})$, $m_j \in \{0, 1, ..., 2^n -1 \} $  and  $j = 1,..., d$. The family $\displaystyle \{ \mathcal {D}_{n} : n=0,1, \ldots \}$ is called the \emph {dyadic filtration} of $Q_0$. Note that each $Q_{n-1}\in \mathcal{D}_{n-1}$ has a unique decomposition $\displaystyle Q_{n-1} = Q^1_n \cup ...\cup Q^{2^d}_n $ where $Q^j_{n} \in \mathcal{D}_n$ for $j=1,...,2^d$. A sequence $\{S_n \}$ of functions $S_n : Q_0 \to \R$ is called a \emph{dyadic martingale} if it verifies the following two conditions: 
\begin{enumerate}
\item Every $S_n$ is constant in each $Q_{n}\in \mathcal{D}_n$, for all $n\geq 0$.
\item For every $n\geq 1$ and each $Q_{n-1}\in \mathcal{D}_{n-1}$, one has 
$$
\fint_{Q_{n-1}} \hspace{-0.2cm}S_{n-1}(x) \, dm_d (x)  = \fint_{Q_{n-1}} \hspace{-0.2cm}S_n (x) \, dm_d (x)  .
$$
\end{enumerate}

We say that the dyadic martingale $\{ S_n \}$ has \emph{uniformly bounded increments} if 
$$ 
\|S\|_{\mathcal{B}} := \sup_n \|S_n - S_{n-1}\|_{\infty}  < \infty .
$$
Observe that if $\{ S_n \}$ is a dyadic martingale with uniformly bounded increments and $S_0 = 0$ then  we have the trivial global bound $\displaystyle \|S_n \|_{\infty} \leq Cn $. However, such trivial bound can be substantially improved for a.e. $x\in Q_0$ according to  the \emph{Law of the Iterated Logarithm}. Indeed, if $\{S_n \}$ is a dyadic martingale in $Q_0$ with uniformly bounded increments and $S_0 = 0$ then there exists $C>0$ depending only on $d$ and $\|S\|_{\mathcal{B}}$ such that 
\begin{equation}\label{lil}
\limsup_{n\to \infty} \frac{|S_n (x)|}{\sqrt{n \log \log n}} \leq C ,
\end{equation}
for a.e. $x\in Q_0$. See \cite{St} for history and an account of the Law of the Iterated Logarithm in different contexts.  

\

\begin{lem}\label{lemmaimproper} 
Let $m\geq 0$ be an integer,  $0< \alpha \leq 1$ and $f\in C^{m, \alpha}(\R^d)$. Let $\sigma$ be a compactly supported (signed) Borel measure on $\R^d$ with finite total variation satisfying \eqref{momentos} for any multiindex $k$ such that $0 \leq |k| \leq [m+\alpha]$. Then, for any cube $Q\subset Q_0$, the following integral
\begin{equation}\label{defmartingale}
S_Q := \int_{0}^1 \fint_{Q} \Delta_{\sigma}f(x,h) \, dm_d (x) \, \frac{dh}{h^{m + \alpha +1}}
\end{equation}
converges. 
\end{lem}

\begin{proof}
Observe that by part c) of Lemmas \ref{lemmacrutial1} and \ref{lemmacrutial2}, we have for $0<h<\ell(Q)/2$,  
\begin{equation}\label{estDelta}
\Big | \fint_{Q} \Delta_{\sigma}f(x,h) \, dm_d (x) \Big |  \leq \begin{cases} C \ell(Q)^{\alpha -1} h^{m+1} \, \, \,  & ,  \, \, 0< \alpha < 1 , \vspace{0.2cm}\\
C  \ell(Q)^{-1}\log \big( \frac{\ell(Q)}{h} \big)  h^{m+2} \, \, \,  & ,  \, \, \alpha = 1 ,
\end{cases}
\end{equation}
where $C = C(m,d, f, \sigma , \alpha ) >0$, so the outer integral is absolutely convergent.
\end{proof}

The key reduction to the martingale setting is provided by the following two lemmas. 
\begin{lem}\label{lemmaincrements}
Let $m\geq 0$ be an integer,  $0< \alpha \leq 1$ and $f\in C^{m, \alpha}(\R^d)$. Assume that $\sigma$ is a compactly supported (signed) Borel measure on $\R^d$ with finite total variation satisfying \eqref{momentos} for any multiindex $k$ such that $0 \leq |k| \leq [m+\alpha]$. Take $Q$, $Q'$ two adjacent subcubes of $Q_0$. Then there exists a constant $C= C(d, m, \alpha, \sigma )>0 $ such that
\begin{equation}\label{increments}
|S_Q  - S_{Q'} | \leq  C \|f\|_{m, \alpha}  .
\end{equation}
\end{lem}

\begin{proof}
Assume that $Q' = Q + \ell(Q)e_j$ for some $j$ with $1 \leq j \leq d$. Then 
\begin{equation}\label{increments}
S_Q - S_{Q'} = \int_0^1 \fint_{Q} \Big[\Delta_{\sigma}f(x,h) - \Delta_{\sigma}f(x+ \ell(Q) e_j, h) \Big]dm_d (x)\, \frac{dh}{h^{m+\alpha + 1}} .
\end{equation}
Denote by $I(h)$ the inner integral in \eqref{increments} and split the outer integral in two terms. Then 
$$
S_Q - S_{Q'} = \int_{0}^{\ell(Q)/2} I(h) \frac{dh}{h^{m+\alpha +1}} + \int_{\ell(Q)/2}^1 I(h) \frac{dh}{h^{m+\alpha +1}} =: A + B.
$$
Let us consider $A$ and $B$ separately. Suppose first that $0< \alpha < 1$. Then, by part c) of Lemma \ref{lemmacrutial1} we have 
$$
|I(h)| \leq C \|f\|_{m, \alpha}\|\sigma\|(\ell(Q))^{\alpha -1}h^{m+1} ,
$$
if $0< h < \ell(Q)/2$, where $C= C(d, m)>0$.  
Therefore 
\begin{equation}\label{A} 
A \leq C \|f\|_{m, \alpha} \, \|\sigma\| \, (\ell(Q))^{\alpha -1} \int_0^{\ell(Q)/2}\frac{dh}{h^{\alpha}} .
\end{equation}
As for B, note that
$$
\Delta_{\sigma}f(x, h) - \Delta_{\sigma}f(x+ \ell(Q) e_j, h) = \int_{\R^d} \Big[f(x+hw) - f(x+ \ell(Q)e_j + hw) \Big]d\sigma (w).
$$
Then part a) of Lemma \ref{lemmacrutial1} implies 
$$
|I(h)| \leq C \|f\|_{m, \alpha} \, \|\sigma\| \, (\ell(Q))^{\alpha} h^m,
$$
so 
\begin{equation}\label{B}
B \leq  C \|f\|_{m, \alpha} \, \|\sigma\| \, (\ell(Q))^{\alpha} \int_{\ell(Q)/2}^1 \frac{dh}{h^{1+\alpha}}
\end{equation}
and the result follows combining \eqref{A} and \eqref{B}. The case $\alpha =1$
 follows analogously from Lemma \ref{lemmacrutial2}.
 \end{proof}

\begin{lem}\label{lemmacompar}
Let $m$, $\alpha$, $f$ and $\sigma$ be as in Lemma \ref{lemmaincrements}. Then, there exists a constant $C = C(d, m, \alpha, \sigma ) >0 $ such that for any subcube $Q\subset Q_0$, any $\varepsilon$ with $\ell(Q)/4 \leq \varepsilon \leq \ell(Q)/2$ and each $x\in Q$ we have 
\begin{equation*}
|S_Q - \Theta^{\sigma}_{\varepsilon}f(x) | \leq C \|f\|_{m, \alpha} \, .
\end{equation*}
 \end{lem}

\begin{proof}
Observe that
\begin{align*}
\displaystyle S_Q - \Theta^{\sigma}_{\varepsilon}f(x) = & \int_0^{\varepsilon} \fint_{Q} \Delta_{\sigma}f(y,h)dm_d (y) \frac{dh}{h^{m+ \alpha +1}} \vspace{0.2cm}\\ \displaystyle  & + \int_{\varepsilon}^1 \fint_{Q} [\Delta_{\sigma}f(y,h) - \Delta_{\sigma}f(x,h)]dm_d (y) \frac{dh}{h^{m+\alpha +1}} \vspace{0.2cm} \\ =: & \,  A + B .
\end{align*}

 In the rest of the lemma $C'$ denotes successive constants of the form $\displaystyle C \|f\|_{m, \alpha} \,  $ where $C = C(d,m, \alpha, \sigma ) >0$. Suppose first that $0< \alpha < 1$. From Lemma \ref{lemmacrutial1} we get
 \begin{align*}
A  & \leq  \,  C' (\ell(Q))^{\alpha -1}\int_0^{\varepsilon} \frac{dh}{h^{\alpha}}  \leq \, \, C', \vspace{0.2cm}\\
B & \leq \,  C' (\ell(Q))^{\alpha } \int_{\varepsilon}^1 \frac{dh}{h^{\alpha +1}} \leq \, \, C' 
\end{align*}
and \eqref{comparison} follows. 
The case $\alpha =1$ follows analogously from Lemma \ref{lemmacrutial2}. 
\end{proof}

\begin{proof}[Proof of Theorem \ref{main}]
As in the previous lemma, $C'$ will denote successive constants of the form $\displaystyle C \|f\|_{m, \alpha} \, $ where $C = C(d,m, \alpha , \sigma) >0$. Given $f$, define, for  every $Q \in \mathcal{D}_n$
\begin{equation}\label{martingale}
S_{n}|_{Q} \equiv S_{Q}
\end{equation}
as in \eqref{defmartingale}. It is clear that \eqref{martingale} defines a dyadic martingale in $Q_0$. From Lemma \ref{lemmaincrements} it follows that $\displaystyle \big | S_{Q} - S_{Q'} \big | \leq C'$ whenever $Q$, $Q' \in \mathcal{D}_n$ are adjacent. This, together with the martingale property, implies that $\{ S_n \}$ has uniformly bounded increments. From Lemma \ref{lemmacompar} it also follows that 
\begin{equation}\label{comparison}
\|S_n  - \Theta^{\sigma}_{\varepsilon}f \|_{\infty }\ \leq C',
\end{equation}
for each $n$ and any $\varepsilon$ such that $2^{-n-2} \leq \varepsilon \leq 2^{-n -1} $. Theorem \ref{main} then follows from \eqref{comparison} and the Law of the Iterated Logarithm (\ref{lil}) applied to the martingale $\{ S_n \}$. 
\end{proof}

\begin{proof}[Proof of Corollary \ref{bounded}]
Since the martingale $\{S_n\}$ has uniformly bounded increments, the set $\{x \in \R^d : \sup_n |S_n (x)| < \infty \}$ has Hausdorff dimension $d$.  Hence the result follows from \eqref{comparison}. 
\end{proof}

%---------------------------------------------------------------------------
\section{The case of Lipschitz functions on the real line. Connection to Calder\'on-Zygmund operators (Proof of Theorem \ref{CZthm})}\label{sect4}

The goal of this section is to proof Theorem \ref{CZthm}.
Our approach is based on Calder\'on-Zygmund theory. For the sake of completeness we recall the fundamental tool that we are going to use.

\begin{thm}{\cite[Theorem 4.4.5]{G}}\label{Th:Grafakos}
	Assume that $\K$ is a locally integrable function on $\R^d \setminus \{0\}$ which satisfies the size condition
\begin{equation}\label{eq:Gra4.4.1}
	\sup_{R>0} \int_{R \leq |x| \leq 2R} |\K(x)| dx =:A_1<\infty,
\end{equation}
the smoothness condition
\begin{equation}\label{eq:Gra4.4.2}
	\sup_{y \neq 0} \int_{ |x| \geq 2|y|} |\K(x-y)-\K(x)| dx =:A_2<\infty,
\end{equation}
and the cancellation condition
\begin{equation}\label{eq:Gra4.4.3}
	\sup_{0<R_1<R_2<\infty} \Big|\int_{R_1 \leq |x| \leq R_2 } \K(x) dx \Big| =:A_3<\infty,
\end{equation}
for certain $A_1, A_2, A_3 >0$. Let $\T_*$ be the maximal singular integral given by
$$\T_*(f)(x)
	:= \sup_{0<\varepsilon <N < \infty} \Big|\K_{\varepsilon,N}*f(x)\Big|,$$
where 
$\K_{\varepsilon,N}(x):= \K(x) \1_{\{\varepsilon \leq | x| \leq N \}}(x).$
Then, $\T_*$ is bounded on $L^p(\R^d)$, $1<p<\infty$, with norm
$$\| \T_* \|_{L^p(\R^d) \to L^p(\R^d)}
	\leq C_d \max\{p,(p-1)^{-1}\}(A_1 + A_2 + A_3).$$
\end{thm}

A locally integrable function $\K$ on $\R^d \setminus \{0\}$ which satisfies the conditions \eqref{eq:Gra4.4.1}, \eqref{eq:Gra4.4.2} and \eqref{eq:Gra4.4.3} is called a Calder\'on-Zygmund kernel. 

The first step consists on writing $\widetilde{\Theta}_{\varepsilon}^{\sigma}$ as a convolution operator.

\begin{lem}\label{Lem:Lip1conv}
Let $0<\varepsilon<1$, $f \in Lip(\R)$ compactly supported and $\sigma$ be a compactly supported (signed) Borel measure on $\R$ with finite total variation satisfying $\sigma(\R)=0$. Then, 
$$\widetilde{\Theta}_{\varepsilon}^{\sigma} (f)(x)
	= K_\varepsilon * f'(x), \quad x \in \R,$$
where
$$K_\varepsilon (t)
	:= \frac{1}{t} \int_{-t/\varepsilon}^{-t} \sigma[s,\infty) ds,
    \quad t \in \R.$$
\end{lem}

\begin{proof}
Fix $x \in \R$. Taking into account that $\sigma(\R)=0$ and an application of the fundamental theorem of calculus give us
\begin{align*}
& \widetilde{\Theta}_{\varepsilon}^{\sigma} (f)(x)
%	& = \int_{\varepsilon}^1 \int_\R f(x+hw) d\sigma(w)\frac{dh}{h^2}
     = \int_{\varepsilon}^1 \int_\R [f(x+hw)-f(x)] d\sigma(w)\frac{dh}{h^2} \\
    & \qquad = \int_{\varepsilon}^1 \int_\R \int_0^{hw}f'(x+t) dt d\sigma(w)\frac{dh}{h^2} \\
    & \qquad = \int_{\varepsilon}^1 \int_0^\infty \int_0^{hw}f'(x+t) dt d\sigma(w)\frac{dh}{h^2}
    - \int_{\varepsilon}^1 \int_{-\infty}^0 \int_{hw}^0 f'(x+t) dt d\sigma(w)\frac{dh}{h^2}.
\end{align*}
Next we use Fubini's theorem, which can be properly justified because $f$ and $\sigma$ have compact support, to write
\begin{align*}
 \widetilde{\Theta}_{\varepsilon}^{\sigma} (f)(x)
 & = \int_0^\infty f'(x+t) \Big(\int_{\varepsilon}^1  \int_{[t/h, +\infty)}  d\sigma(w)\frac{dh}{h^2} \Big) dt \\
 & \qquad - \int_{-\infty}^0 f'(x+t) \Big(\int_{\varepsilon}^1  \int_{(-\infty , t/h)} d\sigma(w)\frac{dh}{h^2}\Big) dt\\
& = \int_0^\infty f'(x+t) \Big(\int_{\varepsilon}^1  \sigma[t/h,\infty) \frac{dh}{h^2} \Big) dt \\
& \qquad - \int_{-\infty}^0 f'(x+t) \Big(\int_{\varepsilon}^1 \sigma(-\infty,t/h) \frac{dh}{h^2}\Big) dt \\
& = \int_\R f'(x+t) \Big(\int_{\varepsilon}^1  \sigma[t/h,\infty) \frac{dh}{h^2} \Big) dt,
\end{align*}
where in the last step we used again that $\sigma(\R)=0$.
Finally, a few change of variables yield to
\begin{align*}
\widetilde{\Theta}_{\varepsilon}^{\sigma} (f)(x)
 & = \int_\R f'(x-t) \Big(\int_{\varepsilon}^1  \sigma[-t/h,\infty) \frac{dh}{h^2} \Big) dt \\
 & = \int_\R f'(x-t) \Big(\frac{1}{t}\int_{-t/\varepsilon}^{-t}  \sigma[s,\infty) ds \Big) dt \\
 & =: (K_\varepsilon * f')(x).
\end{align*}
\end{proof}

Of special interest will be the endpoint kernel $K_0$. 
Next, we analyze its size, smoothness and cancellation properties. 

\begin{lem}\label{Lem:estimacionesK0}
Let $\sigma$ be a (signed) Borel measure on $\R$ with finite total variation supported in the interval
$(-M,M)$, $M>0$, and satisfying  \eqref{momentos} for $k=0,1$. 
Define
$$K_0 (t)
	:= \frac{1}{t} \int_{-\sign(t) M}^{-t} \sigma[s,\infty) ds,
    \quad t \in \R \setminus\{0\}.$$
Then, 
\begin{itemize}
	\item[$a)$] $\displaystyle |K_0 (t)| \leq \frac{2M \|\sigma \|}{|t|}, \quad t \in \R \setminus\{0\},$
    \item[$b)$] $\displaystyle |\partial_t K_0 (t)| \leq \frac{3M \|\sigma \|}{t^2}, \quad t \in \R \setminus\{0\},$
    \item[$c)$] $\displaystyle \sup_{0<a<b<\infty} \Big| \int_{a<|t|<b} K_0 (t) dt \Big| < 3M \|\sigma \|.$
\end{itemize}
\end{lem}

\begin{proof}
To prove $a)$ simply observe that
$$|K_0 (t)|
	\leq \frac{1}{|t|} \int_{-M}^{M} |\sigma[s,\infty)| ds
    \leq \frac{2 M \|\sigma \|}{|t|}, \quad t \in \R \setminus\{0\}.$$

To establish $b)$ we write
\begin{align*}
	\partial_t K_0 (t)
    	& = -\frac{1}{t^2} \int_{-\sign(t) M}^{-t} \sigma[s,\infty) ds - \frac{\sigma[-t,\infty)}{t}, \quad t \in \R \setminus\{0\}.
\end{align*}
Moreover,
$$|\alpha \sigma[\alpha,\infty)| \leq M \|\sigma \|, \quad \alpha \in \R.$$
Hence,
\begin{align*}
|\partial_t K_0 (t)|
    & \leq \frac{2 M \|\sigma \|}{t^2} 
    + \frac{|-t \sigma[-t,\infty)|}{t^2}
     \leq \frac{3 M \|\sigma \|}{t^2} , \quad t \in \R \setminus\{0\}.
\end{align*}

The proof of $c)$ is more subtle because the cancellations of the kernel play an important role. Fix $0<a<b<\infty$. Since $\sigma(\R)=0$ we can write
\begin{align*}
 \int_{a<|t|<b} K_0 (t) dt
	& = \Big(\int_{-b}^{-a} + \int_a^b\Big) \int_{-\sign(t) M}^{-t} \int_s^M d\sigma(w) ds \frac{dt}{t} \\
    & = \int_a^b \int_{t}^{M} \Big(\int_s^M - \int_{-M}^{-s}\Big) d\sigma(w) ds \frac{dt}{t} \\
    & = \int_a^b \int_{t}^{M} \Big( \sigma[s,M) - \sigma(-M,-s]\Big) ds \frac{dt}{t}.
\end{align*}
Notice that when $s \geq M$, $\sigma[s,M)=\sigma(-M,-s]=0$ and then the last integral vanishes.
The analysis is clearer if we consider different cases.\\

\underline{Case 1: $M \leq a$}. Since $M \leq a \leq t \leq s$,  
$$ \int_{a<|t|<b} K_0 (t) dt=0.$$

\underline{Case 2: $a<M<b$}. The integral when $M \leq t \leq b$ vanishes. For the remaining part we apply Fubini's theorem to get
\begin{align*}
 \int_{a<|t|<b} K_0 (t) dt
    & = \int_a^M \Big[\int_{t}^{M} \Big(\int_s^M - \int_{-M}^{-s}\Big) d\sigma(w) ds \Big]\frac{dt}{t} \\
    & = \int_a^M \Big[\Big(\int_{t}^{M} \int_t^w - \int_{-M}^{-t} \int_t^{-w}\Big) ds d\sigma(w)  \Big]\frac{dt}{t} \\
    & = \int_a^M \Big[\int_{t}^{M} (w-t) d\sigma(w) + \int_{-M}^{-t} (w+t) d\sigma(w) \Big]\frac{dt}{t} \\
    & = \int_a^M \Big[- \int_{-t}^{t} w d\sigma(w) - t \sigma[t,M) + t \sigma(-M,t] \Big]\frac{dt}{t},
\end{align*}
where we have used \eqref{momentos} for $k=1$ in the last step. Thus,
\begin{align*}
 \Big|\int_{a<|t|<b} K_0 (t) dt \Big|
    & \leq 3 M \|\sigma \|.
\end{align*}

\underline{Case 3: $a < b \leq M$}. This situation can be essentially treated as in Case 2.
\end{proof}

Now we are in position to prove Theorem \ref{CZthm}.

\begin{proof}[Proof of Theorem \ref{CZthm}.]
% Assume that the support of $\sigma$ is contained in the interval $(-M,M)$, for certain $M>0$. 
By Lemma~\ref{Lem:Lip1conv},
for every $x \in \R$, we have that
\begin{equation*}
 \widetilde{\Theta}_{\varepsilon}^{\sigma} (f)(x) =   I_1(x) + I_2(x),
\end{equation*}
where
$$
I_1 (x) = \int_{|t| \leq \varepsilon M} K_\varepsilon(t) f'(x-t) dt  
$$
and 
$$
I_2 (x) = \int_{|t| > \varepsilon M} K_\varepsilon(t) f'(x-t) dt  .
$$
% The first term above is easy to handle,
% \begin{align*}
%  I_1(x)
%     & \leq \|f'\|_{L^\infty(\R)} \sup_{0<\varepsilon<1}  \int_{|t| \leq \varepsilon M} |K_\varepsilon(t)| dt 
%   \leq \|f'\|_{L^\infty(\R)} 2 M \|\sigma \|, % \quad x \in \R.
% \end{align*}
% Then, it only remains to show that $I_2(x) < \infty$, a.e. $x \in \R$. 
Observe that
$$K_\varepsilon(t)-K_0(t)=0, \quad |t| \geq \varepsilon M, \ 0<\varepsilon <1,$$
and
$$K_0(t)=0, \quad |t| \geq  M.$$
Thus, we can write
\begin{equation*}
I_2(x) = \int_{\varepsilon M < |t| < M} K_0(t) f'(x-t) dt .  
\end{equation*}
Now the first part of Theorem\ref{CZthm} follows from the the easy estimates 
\begin{align*}
& \sup_{0 < \varepsilon < 1}  |\widetilde{\Theta}_{\varepsilon}^{\sigma}f(x) - I_2 (x) | =  
 \sup_{0<\varepsilon<1} |I_1(x)| \leq \\
 &   \leq \|f'\|_{L^\infty(\R)} \sup_{0<\varepsilon<1}  \int_{|t| \leq \varepsilon M} |K_\varepsilon(t)| dt 
  \leq \|f'\|_{L^\infty(\R)} 2 M \|\sigma \|, \quad x \in \R.
\end{align*}
Consider
$$
\T_*(f')(x) = \sup_{0<\varepsilon<N<\infty} \Big| \int_{\varepsilon  < |t| < N} K_0(t) f'(x-t) dt\Big|, 
 \quad x \in \R .
$$
% \begin{align*}
% 	& = \sup_{0<\varepsilon<1} \Big| 
% \int_{\varepsilon M < |t| < M} K_0(t) f'(x-t) dt
% \Big| \\
  %   & \leq \sup_{0<\varepsilon<N<\infty} \Big| \int_{\varepsilon  < |t| < N} K_0(t) f'(x-t) dt\Big| 
%      =:\T_*(f')(x), \quad x \in \R.
% \end{align*}
Since the conclusions of Lemma~\ref{Lem:estimacionesK0} are stronger than \eqref{eq:Gra4.4.1}, \eqref{eq:Gra4.4.2} and \eqref{eq:Gra4.4.3}, Theorem~\ref{Th:Grafakos} implies that
$$\|\T_*(f')\|_{L^2(\R)}
	\lesssim \| f' \|_{L^2(\R)}.$$
Note that $f' \in L^2(\R)$ because it is bounded and compactly supported. In particular, we deduce that
$$I_2(x) \leq \T_*(f')(x) < \infty, \quad \text{a.e. } x \in \R.$$
Therefore
$$
\sup_{0<\varepsilon<1} | \widetilde{\Theta}_{\varepsilon}^{\sigma}f(x)| < \infty \, .
$$

\end{proof}

%%%%%%%%%%%%%%%%%%%%%%%%%%%%%%%%%%%%%%%%%%%%%%%%%%%%%%%%%%%%%%%%%%%
%\bibliographystyle{siam}
%\bibliography{references}
% Use the references.bib file (you can find it using the bottom "project" above). 
% It is much easier to copy and paste references directly from http://www.ams.org/mathscinet in the BibTex format:
%	- Go to the webpage of the article/book in mathscinet
%	- "Select alternative format" --> BibTex

\end{document}